\documentclass[a4paper,12pt]{article}
\usepackage{times, url}
\usepackage{amsmath,amssymb,amsthm,amsfonts}

\theoremstyle{theorem}

\theoremstyle{definition}
\newtheorem*{theorem*}{Theorem}

\begin{document}

\title{A new proof of the AM-GM-HM inequality}
\markright{Abbreviated Article Title}
\author{Konstantinos Gaitanas}

\maketitle
\begin{abstract}
In the current note, we present a new, short proof of the famous AM-GM-HM inequality using only induction and basic calculus.
\end{abstract}
\section{Introduction.}
Perhaps the most celebrated inequality is the AM-GM-HM inequality which states that if we let $\text{AM}=\frac{a_1+\ldots +a_n}{n}$, $\text{GM}=\sqrt[n]{a_1\cdots a_n}$, $\text{HM}=\frac{n}{\frac{1}{a_1}+\ldots +\frac{1}{a_n}}$, then $\text{AM}\ge \text{GM}\ge \text{HM}$ holds if $a_1, \ldots , a_n$ are positive real numbers. Several authors have provided some novel proofs of this inequality (see for example  \cite{Gwanyama}, \cite {Alzer} or \cite{Cauchy} ). The aim of this note is to present a simple new proof which does not seem to appear in the literature.
\section{The proof.}
\begin{theorem*}
For every $a_1, \ldots , a_n>0$, we have \\
\begin{align}\frac{a_1+\ldots +a_n}{n}\ge\sqrt[n]{a_1\cdots a_n}\ge \frac{n}{\frac{1}{a_1}+\ldots +\frac{1}{a_n}}\label{ineq} 
\end{align}
\end{theorem*}
\begin{proof} 
The theorem is obviously true for $n=1$. Suppose now that the theorem holds true for $n-1$, that is $\frac{a_1+\ldots +a_{n-1}}{n-1}\ge\sqrt[n-1]{a_1\cdots a_{n-1}}\ge \frac{n-1}{\frac{1}{a_1}+\ldots +\frac{1}{a_{n-1}}}$. We may write $a_1=a, a_2=b_1\cdot a, \ldots , a_n=b_{n-1}\cdot a$ for some appropriate choice of $b_1, \ldots, b_{n-1}>0$. Cancelling $a$ from all sides we can rewrite \eqref{ineq} in the form $\frac{1+b_1+\ldots+b_{n-1}}{n}\ge \sqrt[n]{b_1\cdots b_{n-1}}\ge \frac{n}{1+\frac{1}{b_1}+\ldots +\frac{1}{b_{n-1}}}$. We first prove the AM-GM inequality: From the induction hypothesis, (multiplying by $n-1$ and adding $1$ at both sides) $1+b_1+\ldots+b_{n-1}\ge 1+(n-1)\sqrt[n-1]{b_1\cdots b_{n-1}}$ holds true, so it suffices to show that $1+(n-1)\sqrt[n-1]{b_1\cdots b_{n-1}}\ge n\sqrt[n]{b_1\cdots b_{n-1}}$. We let $x=\sqrt[n(n-1)]{b_1\cdots b_{n-1}}$, so the last inequality can be written in the form $1+(n-1)x^n\ge nx^{n-1}$. The function $f(x)=1+(n-1)x^n-nx^{n-1}$ has a global minimum at $x=1$, the value $f(1)=0$, so $f(x)\ge 0$ and the inequality holds true for $n$. In order to prove the GM-HM inequality, we use the induction hypothesis (raising both sides to the $\frac{n-1}{n}$) so that $\sqrt[n]{b_1\cdots b_{n-1}}\ge (\frac{n-1}{\frac{1}{b_1}+\ldots +\frac{1}{b_{n-1}}})^{\frac{n-1}{n}}$ holds, and it will be enough to prove that \begin{align}(\frac{n-1}{\frac{1}{b_1}+\ldots +\frac{1}{b_{n-1}}})^{\frac{n-1}{n}}\ge \frac{n}{{1+\frac{1}{b_1}+\ldots +\frac{1}{b_{n-1}}}}\label{Ineq}\end{align} For this purpose, we let $x=\frac{1}{b_1}+\ldots +\frac{1}{b_{n-1}}$ and rewrite \eqref{Ineq} as $(\frac{n-1}{x})^{\frac{n-1}{n}}\ge\frac{n}{1+x}$ which taking logarithms of both sides is equivalent to $(n-1)\ln (n-1)+n\ln (1+x)-(n-1)\ln x-n\ln n\ge 0$. Letting $g(x)=(n-1)\ln (n-1)+n\ln (1+x)-(n-1)\ln x-n\ln n$, it is a routine matter to show that $g(x)$ has a global minimum at $x=n-1$, the value $g(n-1)=0$, so $g(x)\ge 0$ and the theorem is proved for $n$.\\
This completes the proof.
\end{proof}

\end{document}